\documentclass{amsart}
\usepackage[T1]{fontenc}
\usepackage{fullpage}
\usepackage[top=2cm, bottom=4.5cm, left=2.5cm, right=2.5cm]{geometry}
\usepackage{amsmath,amsthm,amsfonts,amssymb,amscd}
\usepackage{lastpage}
\usepackage{enumerate}
\usepackage{fancyhdr}
\usepackage{mathrsfs}
\usepackage{xcolor}
\usepackage{graphicx}
\usepackage{listings}
\usepackage{hyperref}
\usepackage{tikz}
\usepackage{tikz-cd}
\usepackage{commath}
\usepackage{bbm}
\usepackage{multirow}
\usepackage{multicol}
\usepackage{enumitem}
\usepackage{pgfplots}
\usepackage{cleveref}
\usepgfplotslibrary{polar}
\usepgflibrary{shapes.geometric}
\usetikzlibrary{calc}
\usepgfplotslibrary{fillbetween}
\usetikzlibrary{patterns}
\usetikzlibrary{calc}
\usetikzlibrary{decorations.pathmorphing}
\usepackage{quiver}
\usepackage{stmaryrd}
\usepackage{placeins}
\usepackage{amsrefs}
\usepackage{thm-restate}
\usepackage{url}

\newcommand{\p}{\partial}

\newcommand{\R}{\mathbb{R}}
\newcommand{\C}{\mathbb{C}}

\newcommand{\ra}{\rightarrow}

\newcommand{\id}{\text{id}}

\newtheorem{theorem}{Theorem}

\newtheorem{lemma}[theorem]{Lemma}
\theoremstyle{definition}
\newtheorem{definition}[theorem]{Definition}
\newtheorem{prop}[theorem]{Proposition}
\newtheorem*{Acknow}{Acknowledgements}
\newtheorem{remark}[theorem]{Remark}
\newcommand{\A}{\mathcal{A}}

\theoremstyle{definition}
\newtheorem{exmp}[theorem]{Example}

\DeclareMathOperator{\im}{im}

\newcommand{\vdotrows}[2][-15pt]{%
  \multirow{#2}{*}{%
    \vbox {%
      \baselineskip = \dimexpr 2pt\relax 
      \multiply\baselineskip by #2\relax
      \advance\baselineskip by 2pt\relax
      \lineskiplimit = 0pt\relax
      \kern 6pt\relax
      \vskip #1\relax%
      \hbox {.}\hbox {.}\hbox {.}}%
  }%
}

\title[The Hard Lefschetz Theorem on K\"ahler Lie Algebroids]{The Hard Lefschetz Theorem on K\"ahler Lie Algebroids}

\author{Shane Rankin}
\address{Department of Mathematics, University of California, Riverside, CA 92521, USA}
\email{srank007@ucr.edu}

\pgfplotsset{compat=1.18}
\begin{document}

\begin{abstract}
    Compact Kähler manifolds classically satisfy the Hard Lefschetz Theorem, which gives strong control on the underlying topology of the manifold. One expects a similar theorem to be true for Kähler Lie Algebroids, and we show for a certain class of them that this is indeed true, with an added ellipticity requirement. We provide examples of Lie Algebroids satisfying this, as well as an example of a Kähler Lie Algebroid that does not meet this ellipticity requirement and fails to satisfy the Hard Lefschetz condition.
\end{abstract}
\maketitle

\tableofcontents
\section{Introduction}
\subsection{Background}
There are a number of celebrated and named theorems in classical Kähler Geometry demonstrating the excellent behavior of such objects, i.e. The Kodaira and Nakano Vanishing Theorems, The Hard Lefschetz Theorem, The Hodge Index Theorem, etc; though in this paper the Hard Lefschetz Theorem will be our particular focus. Classically on an oriented compact manifold of dimension $n$ we have that 
\begin{equation*}
    \dim(H_{dR}^{k}(M)) \cong \dim(H_{dR}^{n-k}(M)) \qquad \forall k \in \{0,\dots, n\}
\end{equation*} by Poincar\'e Duality, where $H_{dR}^k(M)$ are the deRham cohomology groups of $M$. The Hard Lefschetz Theorem can be viewed as saying that on a compact Kähler manifold, wedging with an appropriate power of the Kähler class is an isomorphism witnessing this phenomenon. Recently, there has been attention towards Kähler Lie Algebroids, see \cite{MR4722602}, where it was shown that both the Kodaira Vanishing Theorem and the Hodge-Riemann Bilinear relations still hold true. 
\subsection{Results}
 The general outline of the proof and paper is as follows: The notions of Kähler Lie algebroids and some of the necessary Kähler identities have been previously established in \cite{MR4722602}; separately, Hodge theory for Lie algebroids has been established in \cite{MR4377306}. In section 2, we recall these results, and show that the two papers results are compatible, though they use different definitions. In section 3, we establish Lie algebroid analogues of some necessary relations from classical Kähler geometry, and use them to prove the main results. In section 4, we discuss impacts of this on the Betti numbers of the algebroid, provide a method of constructing examples, and provide an example of a Kähler Lie algebroid with non-elliptic anchor that fails to satisfy the Hard Lefschetz Theorem. Before we proceed to section two, we present the main results here. Collecting all the necessary conditions to make the aforementioned papers communicate with each other, we say that a K\"ahler Lie Algebroid $(\A,\rho ,[\cdot , \cdot], h)$ over a manifold $M$ is \textit{Hodge-admissible} if the following conditions hold:
    \begin{enumerate}
        \item $M$ is compact, orientable, and closed.
        \item $\A$ is elliptic.
        \item $\A$ carries an integrating section.
    \end{enumerate}
With such a Lie algebroid, we regain a Dolbeault-type decomposition of the Lie algebroid cohomology as well as the Hard Lefschetz Theorem.
\begin{restatable}{thm}{decomp}
     Let $(\A,\rho, [ \cdot, \cdot],h)$ be a Hodge-admissible K\"ahler Lie algebroid. Then the harmonic forms are bigraded, that is
    \begin{equation*}
        \mathcal{H}^k = \bigoplus_{p+q=k}\mathcal{H}^{p,q},
    \end{equation*}
    where $\mathcal{H}^{p,q} = \ker(\Delta) \cap\Omega_\A^{p,q}$ is finite-dimensional. We also have that
    \begin{equation*}
        H_\A^k(M;\C) \cong \bigoplus_{p+q=k} \mathcal{H}^{p,q},
    \end{equation*}
    where $H_\A^k(M;\C)$ is the Lie Algebroid cohomology.
\end{restatable}

\begin{restatable}{thm}{HLT}
     If $(\A,\rho ,[\cdot, \cdot],h)$ is a Hodge-admissible K\"ahler Lie Algebroid of rank $2m$, then $\A$ satisfies the $\A$-Hard Lefschetz Theorem. That is, the map
    \begin{align*}
        [L]^k:H^{m-k}_\A(M) &\ra H_\A^{m+k}(M) \\
        [\alpha] &\mapsto [\omega^k\wedge \alpha]
    \end{align*}
    is an isomorphism for all $k \geq 0$, where $\omega$ is the $\A$-Kähler form associated to $h$.
\end{restatable}

\begin{restatable}{thm}{Formality}
    If $(\A,\rho ,[\cdot, \cdot],h)$ is a Hodge-admissible K\"ahler Lie algebroid, then $\A$ is formal.
\end{restatable}
 A classical corollary of the Hard Lefschetz Theorem is that the odd Betti numbers of a compact Kähler manifold are even. If we assume that the relevant pairing $I$ provided by Poincar\'e Duality carries over to the Algebroid setting, then we get an analogous result:
\begin{restatable}{thm}{betti}
     If $(\A,\rho ,[\cdot, \cdot],h)$ is a Hodge-admissible K\"ahler Lie algebroid such that $I$ is non-degenerate, then the odd degree cohomology groups are even dimensional.
\end{restatable}
This is discussed in Section 4.1, and provides an obstruction for a Kähler Lie Algebroid to satisfy the Hard Lefschetz Theorem.
\begin{Acknow}
    The author would like to thank Dr. Yi Lin for the many helpful suggestions and conversations surrounding this result. In addition, many details surrounding the results were clarified and illuminated by conversations at Gone Fishing 2025 for which the author expresses gratitude. Finally, the author wishes to express gratitude to the anonymous  referee for the helpful suggestions; all of which greatly improved the article in clarity, flow, and simplicity.
\end{Acknow}

\section{Elliptic Lie Algebroids}

\subsection{Complex Lie Algebroids}
Throughout the paper, let $(M,\tilde{g})$ denote a compact, closed, orientable, Riemannian manifold of real dimension $n$. All vector bundles are assumed to be complex unless otherwise indicated. The following definition is found in \cite{MR4722602}, where the author makes reference to \cite{MR3483864} as well as \cite{weinstein2006integrationproblemcomplexlie}.
\begin{definition}\cite[Definition 1]{MR4722602}
    A \textit{Complex Lie Algebroid on $M$} is a tuple $(\A, \rho, [\cdot, \cdot], J)$ where:
    \begin{itemize}
        \item $\A$ is a complex vector bundle on $M$.
        \item $\rho:\A \ra TM_\C$ is a vector bundle morphism called the \textit{anchor of $\A$}.
        \item $[\cdot , \cdot]$ is a bracket on $\Gamma(\A)$ giving the space of sections the structure of a Lie algebra.
        \item $J: \A \ra \A$ is the complex structure of $\A$, regarded as a bundle morphism over the identity such that $J^2=-\id_\A$.
    \end{itemize}
    such that:
    \begin{enumerate}
        \item $\rho$ induces a Lie algebra homomorphism on sections, i.e. for all $X,Y \in \Gamma(\A)$, we have
        \begin{equation*}
            \rho([X,Y]) = [\rho(X), \rho(Y)]_{TM_\C}.
        \end{equation*}
        \item For all $X,Y \in \Gamma(\A)$ and $f \in C^\infty(M)$, we have
        \begin{equation*}
            [X,fY] = f[X,Y]+(\rho(X)f)Y.
        \end{equation*}
        \item For all $X,Y \in \Gamma(\A_\C)$
    \begin{equation*}
        \mathcal{N}_{J}(X,Y) := [JX,JY] - [X,Y] - J([JX,Y]+[X,JY]) \equiv 0
    \end{equation*}
    That is, $J$ is necessarily integrable.
    \end{enumerate}
\end{definition}
 Splitting $\A_\C$ into its eigenbundle decomposition with respect to the complexified $J$ allows us to write out
\begin{equation*}
    \A_\C = \A^{1,0} \oplus \A^{0,1},
\end{equation*}
where 
\begin{equation*}
    \Gamma(\A^{1,0}) = \{X -iJ_\A X | X \in \Gamma(\A) \}, \qquad  \Gamma(\A^{0,1}) = \{X +iJ_\A X | X \in \Gamma(\A) \}.
\end{equation*}
This splitting extends to the dual bundle $\A^*_\C = (\A^{1,0})^*\oplus (\A^{0,1})^*$. We then set
\begin{equation*}
    \Omega^{p,q}_\A(M;\C) = \Gamma \left(\bigwedge^p(\A^{1,0})^* \otimes \bigwedge^q(\A^{0,1})^* \right)
\end{equation*}
which splits the standard complex into
\begin{equation*}
    \Omega_\A(M;\C) =\bigoplus_k \Omega_\A^k(M;\C)  = \bigoplus_{p,q}\Omega^{p,q}_\A(M;\C)
\end{equation*}
Where $\Omega_\A^k(M;\C) = \Gamma(\bigwedge^k\A^*_\C)$. Much like in the classical setting of the tangent bundle, $\A$ carries a natural differential defined as follows. For $\eta \in \Omega^k_\A(M;\C)$ and $X_i \in \Gamma(\A_\C)$ we define $d_\A$ as
\begin{align*}
    d_\A \eta (X_1,\dots ,X_{k+1}) :&= \sum_{1 \leq i \leq k+1} (-1)^{i+1}\rho(X_i) \eta (X_1,\dots ,\widehat{X}_i, \dots ,X_{k+1})\\
    &+ \sum_{i<j} (-1)^{i+j} \eta([X_i,X_j], X_1, \dots, \widehat{X_i}, \dots ,\widehat{X_j}, \dots, X_{k+1}).
\end{align*}
The \textit{deRham-Chevalley-Eilenberg Cohomology of $\A$} is defined as the cohomology of the complex
\[\begin{tikzcd}[ampersand replacement=\&]
	\dots \& {\Omega_\A^{k-1}(M;\C)} \& {\Omega_\A^k(M;\C)} \& {\Omega_\A^{k+1}(M;\C)} \& \dots
	\arrow["{d_\A}", from=1-1, to=1-2]
	\arrow["{d_\A}", from=1-2, to=1-3]
	\arrow["{d_\A}", from=1-3, to=1-4]
	\arrow["{d_\A}", from=1-4, to=1-5]
\end{tikzcd}\]
that is
\begin{equation*}
    H_\A^k(M;\C) : =\frac{\ker(d_\A: \Omega^k_\A(M;\C) \ra \Omega_\A^{k+1}(M;\C))}{\im(d_\A:\Omega_\A^{k+1}(M;\C) \ra \Omega_\A^k(M;\C))}.
\end{equation*}
 As in the case of classical complex geometry, the differential also splits into four separate components, though we'll only consider analogues of the standard two for the following reason.
\begin{theorem}[\cite{MR3483864},Theorem 2.1]
    For a Complex Lie Algebroid $(\A,\rho,[\cdot, \cdot],J)$, we have that
         $d_\A \Omega^{p,q}_\A(M) \subset \Omega^{p+1,q}_\A(M) \oplus \Omega ^{p,q+1}_\A(M)$.
\end{theorem}
This theorem ensures that our differential decomposes 
\begin{equation*}
    d_\A = \p_\A + \overline{\p_\A}:\Omega^{p,q}_\A(M) \ra \Omega_\A^{p+1,q}(M) \oplus \Omega_\A^{p,q+1}(M),
\end{equation*}
and comparing types we have that
\begin{equation*}
    \p_\A^2 =0, \qquad \p_\A \overline{\p_\A} = -\overline{\p_\A} \p_\A ,\qquad (\overline{\p_\A}) ^2=0.
\end{equation*}
One can define the analogue of the Dolbeault complex in this setting, though we will not need it for our purposes.
\subsection{Ellipticity} On a compact and
closed Hermitian manifold, the description of the deRham and Dolbeault cohomologies in terms of harmonic forms is possible because of the ellipticity of the ensuing differential complexes. The following definition of ellipticity for Lie algebroids can be found in \cite[Definition 1.15]{MR4377306}.
\begin{definition}
    A Complex Lie Algebroid $(\A, \rho ,[\cdot, \cdot])$ is \textit{Elliptic} if 
    \begin{equation*}
        \rho(\A) + \overline{\rho(\A)} = TM_\C.
    \end{equation*}
\end{definition}
\begin{exmp}
    Let $X$ be a real manifold with an almost complex structure $J$, and consider 
    \begin{equation*}
        \A := \{X-iJX | X \in TX\} \subset TX_\C.
    \end{equation*}
    The inclusion provides a natural candidate for an anchor map, and the complexified Lie bracket restricts to a bracket on $\A$ if $J$ is integrable. In this case, we have that 
    \begin{equation*}
        \rho(\A) \oplus \overline{\rho(\A)} = T^{1,0}X \oplus T^{0,1}X = TX_\C.
    \end{equation*}
    Note that in this case the differential corresponding to $T^{1,0}X$ is $d_\A = \p$ which is indeed an elliptic operator. 
\end{exmp}
\begin{exmp}
    Let $\mathfrak{g}=\C^n$ regarded as an abelian Lie Algebroid over a point $M=*$. Define $J:\mathfrak{g} \ra \mathfrak{g}$ by $Jv=iv$, the standard (integrable) complex structure on $\C^n$. This is trivially elliptic as $\rho(\mathfrak{g}) \oplus \overline{\rho(\mathfrak{g})}=TM_\C$.
\end{exmp}
Let $(\A,\rho, [\cdot,\cdot])$ be a Lie algebroid over a smooth manifold $M$. Let $h$ be a Hermitian Metric on $\A$ and $\tilde{g}$ be a Riemannian metric on $M$. Let $d_\A^\dagger$ denote the adjoint of $d_\A$ with respect to the $L^2$-inner product,
\begin{equation*}
    \langle \alpha, \beta \rangle_{L^2} := \int_M h(\alpha, \beta) d\text{vol}_{\tilde{g}},
\end{equation*}
where we extend $h$ to a hermitian inner product on $\Omega_\A^k(M;\C)$ by defining it on decomposable elements $\alpha_1 \wedge \dots \wedge \alpha_k, \beta_1 \wedge \dots \wedge \beta_k$ as
\begin{equation*}
    h(\alpha_1 \wedge \dots \wedge \alpha_k, \beta_1 \wedge \dots \wedge \beta_k) := \det(h(\alpha_i,\beta_j)).
\end{equation*}
\begin{definition}
    Let $d_\A,d_\A^\dagger$ be defined as above. Then the \textit{Laplacian} is defined as
    \begin{equation*}
        \Delta = d_\A d_\A^\dagger + d_\A ^\dagger d_\A
    \end{equation*}
    which is a degree 0, second order differential operator on $\Omega_\A^k(M;\C)$ for all $k \geq 0$.
\end{definition}
\begin{definition}
    We say that $\alpha \in \Omega_\A^k(M;\C)$ is \textit{harmonic} if $\alpha \in \ker(\Delta)$ or equivalently if $\alpha \in \ker(d_\A) \cap \ker(d_\A^\dagger)$. The space of harmonic elements in $\Omega^k(M)$ will be denoted $\mathcal{H}^k$.
\end{definition}

\begin{theorem}[van der Leer Dur\'an]\label{HodgeDecomp}
    If $\A$ is elliptic, then for all $k \geq 0$  we have that
    \begin{equation*}
        \Gamma( \bigwedge^k \A^*) = \mathcal{H}^k \oplus \im(d_\A) \oplus \im(d_\A^\dagger).
    \end{equation*}
    Moreover we have that
    \begin{equation*}
        \ker(d_\A) \cap \ker(d_\A^\dagger)=\mathcal{H}^k \cong H^k_\A(M;\C),
    \end{equation*}
    which are all finite dimensional.
\end{theorem}
\begin{proof}
    See \cite[Theorem 1.45]{MR4377306}.
\end{proof}

\section{Kähler Lie Algebroids}
\subsection{Integration Compatibility} 
The following definition of a Kähler Lie algebroid can be found in \cite[Definition 2]{MR4722602}:
\begin{definition}
    A \textit{Kähler Lie algebroid} $(\A,\rho,[\cdot ,\cdot],h)$ is a complex Lie algebroid equipped with a hermitian metric $h$ on $\A$ whose associated $(1,1)$-$\A$-form $\omega$ is $d_\A$-closed. That is, 
    \begin{equation*}
        h= g-i\omega.
    \end{equation*} We refer to $\omega$ as an \textit{$\A$-Kähler form}.
\end{definition}
\begin{exmp} \label{TangentBundle}
   Let $X$ be a Kähler manifold, and identify $TX$ with $T^{1,0}X$ via the complex linear isomorphism $X \mapsto \frac{1}{2}(X-iJX)$, where we regard $TX$ as a complex vector bundle, with complex structure provided by $J$. Earlier, we discussed that this is an elliptic Lie algebroid. The hermitian metric from the Kähler structure also serves as a hermitian metric here, with the (1,1)-$\A$-form just the typical Kähler form. Note then that $d_\A \omega = \p \omega$, which does vanish as $d\omega=0$ implies that $\p \omega=0$ as well.
\end{exmp}
\begin{exmp}\label{LieAlgebra}
    We previously discussed the abelian Lie algebra $\C^n$, regarded as a Lie algebroid over a point, and that this was an elliptic Lie Algebroid.
   In this case, the standard hermitian inner product on $\C^n$ gives rise to a Kähler Lie algebroid structure. In this case, the imaginary part of the hermitian metric resembles the standard symplectic form on $\R^{2n}$. It is indeed closed as $d \equiv 0$ in this example, though it is not exact for the same reason. 
\end{exmp}
The classical Kähler identities still hold in this setting, as well as a number of other familiar properties as shown in \cite{MR4722602}. In order to see these in the context of the ellipticity already laid out, we spend a bit of time showing that the notions of integration defined in \cite{MR4722602} agree with what we have already done. The central issue that occurs is the following: it would be convenient to define adjoints to the algebroid exterior derivative against something like the following map
 \begin{equation*}
     \langle \alpha, \beta \rangle = \int_M \alpha \wedge \overline{\star \beta} = \int_M h(\alpha,\beta) \frac
     {\omega_\A^m}{m!},
 \end{equation*}
 where $\frac {\omega_\A^m}{m!} \in \bigwedge^{top}\A^*$. Unfortunately, $\frac{\omega_\A^m}{m!}$ is not a volume form on $M$ so this fails. Moreover, we have already defined integration of algebroid forms using the Riemannian metric on $M$ in order to discuss the Hodge decomposition. This can be remedied though, and to do so we introduce the following definition, which is used in defining integration in \cite{MR4722602}:
 \begin{definition}
     An \textit{integrating section} for a Complex Lie Algebroid $\A \ra M$ is a nonvanishing global section $\eta \in \Gamma(\bigwedge^{top}\A \otimes \bigwedge^{top}T^*M_\C)$.
 \end{definition}
To explain the name, let $\langle , \rangle$ denote the map
\begin{align*}
    \langle \cdot , \cdot \rangle: \bigwedge^{top} \A^*_p \otimes \left(\bigwedge^{top}\A_p \otimes_{C^\infty(M)}\bigwedge^{top}T_p^*M_\C \right) &\ra \bigwedge^{top}T_p^*M_\C \\
    \xi \otimes ( X\otimes \mu) &\mapsto  (\xi( X) )\mu,
\end{align*}
where $\xi( X)  \in C^\infty(M)$ denotes the pairing between $\A$ and its dual. Using this we can define an (a priori different) integral.
    If $\A \ra M$ is a Lie Algebroid over a compact oriented manifold $M$, and $\xi \in \bigwedge^{top}\A^*$. Then we define the integral of $\xi$ to be
    \begin{equation*}
        \int_M \langle\xi, \eta \rangle,
    \end{equation*}
where $\eta$ is an integrating section of $\A$. Note that in particular, if we are working with a fixed element of $\xi \in \bigwedge^{top}\A^*$ then we can modify $\eta$ to give us something particular:
\begin{definition}
    Let $(\A,\rho ,[\cdot, \cdot])$ be a Lie Algebroid over an oriented Riemannian Manifold $(M,\tilde{g})$ with an integrating section $\eta$, and fix $\xi \in \bigwedge^{top}\A^*$ non-vanishing. Then $\eta$ is called \textit{$\xi$-normal} if $\langle  \xi , \eta \rangle = d \text{vol}_{\tilde{g}}$.
\end{definition}
\begin{lemma}\label{normalization}
    Let $\A \ra M$ be a Lie Algebroid with an integration section $\eta$, fix $\xi \in \bigwedge^{top}\A^*$ non-vanishing, and let $\tilde{g}$ denote the Riemannian metric on $M$. Then we can always choose an integrating section to be $\xi$-normal.
\end{lemma}
\begin{proof}
Let $d\mathrm{vol}_{\tilde{g}}$ be the standard volume form, and note that as $\xi,\eta$ are both non-vanishing, we have that 
\begin{equation*}
    \langle \xi ,\eta \rangle = f d\mathrm{vol}_{\tilde{g}}
\end{equation*}
for some non-vanishing $f$. Then choosing $\chi:= \frac{1}{f}\eta$, we have that $\chi$ is an $\xi$-normal integrating section.
\end{proof}
In light of this, from now on we assume all integrating sections $\eta$ are $\omega^m/m!$-normal where $\omega$ is the Kähler form associated to the chosen hermitian metric $h$ on $\A$, and $m$ is half of the real rank of $\A$. With this assumption in place, the two notions of integration defined agree, that is for $\alpha ,\beta \in \Omega_\A^k(M;\C)$ we have:
\begin{equation*}
    \int_M\langle h(\alpha,\beta) \frac{\omega_\A^m}{m!}, \eta \rangle = \int_M h(\alpha,\beta) d\mathrm{vol}_{\tilde{g}}.
\end{equation*}
\subsection{Kähler Identites}
 On a Kähler Lie algebroid, if we denote by $\omega$ the $\A$-Kähler form we can define a map
\begin{align*}
    L :\Omega^{p,q}_\A(M;\C) &\ra \Omega^{p+1,q+1}_\A(M;\C) \\
    \alpha &\mapsto \omega \wedge \alpha.
\end{align*}
Classically this operator is part of an $\mathfrak{sl}_2(\C)$ representation, and the same is true in the Algebroid setting. Let $\Lambda$ be the formal adjoint to $L$ as defined in \cite{MR4722602}. If we define $H$ as
\begin{equation*}
    H|_{\Omega^{p,q}_\A(M;\C)}=(p+q-m)\id_{\Omega^{p,q}_\A(M;\C)} 
\end{equation*}
where $\mathrm{rank}(\A)=2m$, and extend it to all forms linearly, we have the following:
\begin{prop}
    Let $(\A, \rho , [\cdot, \cdot],h)$ be a Kähler Lie algebroid, $L, \Lambda$, and $H$ be as above. Then the assignment
    \begin{equation*}
        \rho(x) = L, \quad \rho(y)=\Lambda, \quad \rho(z)=H,
    \end{equation*}
    is an $\mathfrak{sl}_2(\C)$ representation on the space of $\A$-forms, where $\{x,y,z\}$ is the basis of $\mathfrak{sl}_2(\C)$ such that $[x,y]=z, [x,z]=-2x, [y,z]=2y$.
\end{prop}
\begin{proof}
    The fact that $[L,\Lambda]=H$ is the content of \cite[Lemma 5]{MR4722602}. The other two relations follow readily from the definition of $H$.
\end{proof}

\begin{theorem}\label{HuKahlerRelations}
    Suppose $(\A,\rho ,[,],h)$ is a K\"ahler Lie Algebroid, the the following equalities hold
    \begin{align*}
        [\p_\A^\dagger ,L] = -i\overline{\p}_A , \qquad
        [\overline{\p}_\A^\dagger, L] = i\p_\A,
    \end{align*}
    where adjoints are taken with respect to the $L^2$-inner product previously discussed.
\end{theorem}
\begin{proof}
    See \cite[Theorem 1]{MR4722602}.
\end{proof}
On a Complex Lie algebroid, one can define more sensitive notions of the Laplacian than discussed in the previous section using the $\p_\A$ and $\overline{\p_\A}$ operators:
\begin{align*}
    \Delta_\p &:= \p_\A \p_\A^\dagger + \p_\A^\dagger \p_\A, \\
    \Delta_{\overline{\p}} &:= \overline{\p}_\A \overline{\p}_\A^\dagger + \overline{\p}_\A^\dagger\overline{\p}_\A. 
\end{align*}
On a K\"ahler Lie algebroid not only do these two maps agree, but they agree up to a nonzero constant with the standard Laplacian. 
\begin{prop}\label{AllLaplaciansAreSame}
On a K\"ahler Lie Algebroid $(\A, \rho, [\cdot, \cdot], h)$ we have that $\Delta = \Delta_\p + \Delta_{\overline{\p}}$. Moreover, we have that $\Delta_\p = \Delta_{\overline{\p}} = \frac{1}{2}\Delta$.
\end{prop}
\begin{proof}
    This follows identically to the classical proofs as found in \cite[page 115-116]{MR507725}.
\end{proof}
There are a few other relations that will be convenient to have on hand for upcoming proofs.
\begin{lemma}
    On a Kähler Lie algebroid $(\A, \rho, [\cdot, \cdot],h)$, the following relations hold:
    \begin{enumerate}
        \item $[L,d_\A]=0$,
        \item $[L,\overline{\p}_\A]=0$,
        \item $[L,\Delta_{\overline{\p}}]=0$, 
        \item $[L,\Delta]=0$,
    \end{enumerate}
    where the commutators are regarded in the graded sense.
\end{lemma}
\begin{proof}
    The first relation follows from the fact that $\omega$ is closed, and implies the second relation.  To see the third relation, note that we have
    \begin{align*}
        \Delta_{\overline{\p}}L \alpha &= \overline{\p}_\A\overline{\p}^\dagger_\A L \alpha +  \overline{\p}^\dagger_\A\overline{\p}_\A L\alpha  \\
        &=\overline{\p}_\A(L \overline{\p}_\A^\dagger +i \p_\A)\alpha + \overline{\p}_\A^\dagger L\overline{\p}_\A\alpha \\
        &=L\overline{\p}_\A \overline{\p}_\A^\dagger \alpha + i\overline{\p}_\A\p_\A \alpha + (L \overline{\p}^\dagger_\A+ i\p_\A) \overline{\p}_\A\alpha \\
        &= L\overline{\p}_\A \overline{\p}_\A^\dagger \alpha + i \overline{\p}_\A \p_\A \alpha + L \overline{\p}_\A^\dagger \overline{\p}_\A \alpha + i \p_\A \overline{\p}_\A \alpha \\
        &=L \Delta_{\overline{\p}}\alpha.
    \end{align*}
    The final relation then follows from Proposition \ref{AllLaplaciansAreSame}, as we have
    \begin{equation*}
        [L,\Delta] =2[L,\Delta_{\overline{\p}}]=0.
    \end{equation*}
\end{proof}
 At this point, there are a a few necessary conditions for all the previous results to both work and agree with each other, so for convenience sake we'll introduce a new definition to combine them all:
\begin{definition}
    We say a K\"ahler Lie Algebroid $(\A,\rho ,[\cdot , \cdot], h_\A)$ over a manifold $M$ is \textit{Hodge-admissible} if the following conditions hold:
    \begin{enumerate}
        \item $M$ is compact, orientable, and closed.
        \item $\A$ is elliptic.
        \item $\A$ carries an integrating section.
    \end{enumerate}
\end{definition}
\begin{remark}
    In light of Lemma \ref{normalization}, any integrating section can be adjusted to be $\omega^m/m!$-normal, where $\omega$ is the $\A$-Kähler form associated to the chosen Kähler metric. As previously remarked, all integrating sections will be assumed to be normalized this way.
\end{remark}

 \decomp*
 
\begin{proof}
    The fact that $\mathcal{H}^k$ is finite-dimensional and is isomorphic to $H_\A^k(M;\C)$ follows from Theorem \ref{HodgeDecomp}. As for the bigraded sum, first note that if $\alpha \in \mathcal{H}^k$, then we can decompose it as $\alpha = \sum_{p+q=k}\alpha_{p,q}$ with $\alpha_{p,q}\in \Omega^{p,q}_\A(M)$. Using Proposition \ref{AllLaplaciansAreSame}, we know that the Laplacian preserves bidegree, and so vanishing of the left hand side of this equality implies the independent vanishing of the right-hand side's summands. 
\end{proof}
The proof of the following theorem is adapted from \cite[Theorem 9.46]{MR4729636}.
\HLT*

\begin{proof} Fix $0 \leq k \leq m$ and let us proceed to check bijectivity directly: \newline
\textit{Injectivity:} Suppose that $[\alpha] \in \ker([L]^k)$, which by Theorem \ref{HodgeDecomp} we may assume $\alpha$ was chosen to be a harmonic representative. Then $L^k \alpha=d\beta$ for some $\beta$, and since $\alpha$ is harmonic and $[L,\Delta]=0$, we must have that $d\beta$ is harmonic as well. This implies $d\beta \in \im(d) \cap \ker(d^\dagger)$, which by Theorem \ref{HodgeDecomp} have trivial intersection, hence $d\beta=0$, and so $L^k(\alpha)=0$. Since the space of $\A$-forms form an $\mathfrak{sl}_2(\C)$ module, $L^k:\Omega_\A^{m-k}(M;\C)\ra \Omega^{m+k}_\A(M;\C)$ acts injectively forcing $\alpha=0$.  \newline 
\textit{Surjectivity:}
Fix $[\gamma] \in H_\A^{m+k}(M;\C)$, and again using Theorem \ref{HodgeDecomp} assume that $\gamma$ is a harmonic representative of its class. Using the $\mathfrak{sl}_2(\C)$ structure, we have that $L^k:\Omega_\A^{m-k}(M;\C)\ra \Omega^{m+k}_\A(M;\C)$ is a surjection, and so there is an $\eta$ such that $L^k \eta = \gamma$. We claim that such an $\eta$ itself must be harmonic. Note that since $ L^k(\Delta\eta) = \Delta (L^k\eta) = \Delta \gamma =0$, we have that $\Delta \eta \in \ker (L^k)$, but using the $\mathfrak{sl}_2(\C)$ structure we have that $L^k$ is an injection, forcing $\Delta \eta =0$. Since $\eta$ is harmonic, it represents a cohomology class, and we have that $[L]^k([\eta])=[\gamma]$.
\end{proof}
\begin{remark}
    There are two other methods of proof in this setting. One can introduce the notion of symplectic harmonic forms in the algebroid sense, and argue that the Kähler condition ensures this agrees with standard harmonicity. From this, one can then deduce the Hard Lefschetz theorem, as it is equivalent to the existence of symplectic harmonic representatives in every cohomology class. Alternatively, one can show that that the $\mathfrak{sl}_2(\C)$ representation passes to cohomology by using the harmonic representatives, and then invoke the structure theory of finite-dimensional $\mathfrak{sl}_2(\C)$ representations.
\end{remark}
In the same vein of cohomological rigidity, we have a generalization of the classical statement that compact Kähler manifolds are formal. Recall that a manifold is called \textit{formal} if it's deRham complex is quasi-isomorphic to it's cohomology, that is $(\Omega^\bullet(M),d) \simeq(H_{dR}^\bullet(M),0)$. Likewise, we say that a Lie algebroid is \textit{formal} if the algebroid deRham complex is quasi-isomorphic to it's own cohomology.
\Formality*
\begin{proof}
    \cite[Corollary 78]{RANKIN2026102324} states that any symplectic Lie algebroid for which the maps $[L]^k$ are all isomorphisms is necessarily formal. Theorem 1.2 states that on a Hodge-admissible Kähler Lie algebroid, these maps are necessarily isomorphisms.
\end{proof}

\section{Consequences}
\subsection{Betti Parity}
A classic consequence of the Hard Lefschetz Theorem is that the odd Betti numbers of a compact Kähler manifold are even. A similar fact is true for some Hodge-admissible Kähler Lie algebroids. We begin by defining a bilinear paring $I$ on $H^{m-k}_\A(M;\C)$ as
\begin{align*}
    I:H_\A^{m-k}(M;\C) \otimes H_\A^{m-k}(M;\C) &\ra \C \\
    ([\alpha],[\beta]) &\mapsto \int_M \langle L^k(\alpha) \wedge \beta,\eta \rangle.
\end{align*}
Classically, Poincar\'e duality ensures that this bilinear form is nondegenerate, however the Lie Algebroid analogue of the perfect pairing in Poincar\'e duality fails at times, see \cite[Examples 5.7 and 5.8]{evens1996transversemeasuresmodularclass}. In the case that $I$ is nondegenerate, we have the following theorem:
\betti*
\begin{proof}
    By construction of $I$, we have that 
    \begin{align*}
        I([\alpha],[\beta]) &=\int_M \langle L^k(\alpha) \wedge \beta , \eta \rangle \\
        &=\int_M \langle \omega^k \wedge \alpha \wedge \beta , \eta \rangle \\
        &=(-1)^{(k-m)^2}\int_M\langle \omega^k \wedge \beta \wedge \alpha , \eta \rangle \\
        &=(-1)^{(k-m)^2} I([\beta],[\alpha]).
    \end{align*}
    When $k-m$ is odd, $I$ is then a nondegnerate, alternating form, forcing the dimension of $H_\A^{m-k}(M;\C)$ to be even. 
\end{proof}
A natural question is what are sufficient and necessary conditions on a Lie Algebroid for $I$ to be nondegenerate. There are a number of examples satisfying this in \cite[section 5]{evens1996transversemeasuresmodularclass}, though at the current time the author is unaware of necessary or sufficient conditions for this Poincar\'e Duality-esque statement to hold in full.
\subsection{Examples}
\begin{prop}
    If $(\A_j \ra M_j,\rho_j, g_j, \omega_j, J_j)$ for $j =1,2$ are Kähler Lie Algebroids with Kähler metrics $h_j=g_j-i\omega_j$, then so too is $(\A_1 \times \A_2 \ra M_1 \times M_2, \rho_1 \oplus \rho_2, g= \pi_1^*g_1 + \pi_2^*g_2, J = J_1 \oplus J_2, \omega = \pi_1^* \omega_1 + \pi_2^* \omega_2)$.
\end{prop}
\begin{proof}
    The fact that the product is again a symplectic Lie algebroid follows from \cite[Proposition 63]{RANKIN2026102324}. Next, the almost complex structure $J$ on $\A_1 \times \A_2$ defined by $J(a_1,a_2):=(J_1a_1,J_2a_2)$ is integrable since $N_J = N_{J_1}+N_{J_2}=0$ since both components are Kähler. Finally, we have compatibility as if $g:\pi_1^*g_1 + \pi_2^* g_2$, then $g( \cdot ,  \cdot) = \omega ( \cdot , J \cdot)$ as both components are compatible triples:
    \begin{align*}
        \omega((a_1,b_1), J(a_2,b_2)) &= \pi_1^*\omega_1((a_1,b_1), J(a_2,b_2)) + \pi_2^*\omega_2((a_1,b_1), J(a_2,b_2)) \\
        &= \omega_1(a_1, J_1 a_2)+ \omega_2(b_1,J_2b_2) \\
        &= g_1(a_1,a_2) + g_2(b_1,b_2) \\
        &= \pi_1^*g_1((a_1,b_1), J(a_2,b_2)) + \pi_2^* g_2 ((a_1,b_1), J(a_2,b_2)) \\
        &= g((a_1,b_1), (a_2,b_2)),
    \end{align*}
     as such $h := g-i\omega$ is a Kähler metric on $\A_1 \times \A_2$.
\end{proof}
We can use this with previous examples to get new one:
\begin{exmp}
    Consider the Kähler Lie algebroid in Examples \ref{TangentBundle} and \ref{LieAlgebra}.Then their product $\A :=\C^n \times T^{1,0}X $ is a Kähler Lie Algebroid over $X \times \{*\} \cong X$ where $X$ is a Kähler manifold of real dimension $2m$. The bracket is the standard product bracket, and the anchor map is given by projection onto the second factor. Now we claim that not only is this Kähler, but Hodge-admissible as well. First, we have that $X$ is compact, closed, and oriented as it's a compact Kähler manifold. Next, we have that the algebroid is elliptic since
    \begin{equation*}
        \rho(\A) + \overline{\rho(\A)} = T^{1,0}X+\overline{T^{1,0}X} =T^{1,0}X+T^{0,1}X= TX_\C.
    \end{equation*}
    Finally, we must check that this carries an integrating section, so we need to find a non-vanishing section of $\bigwedge^{\text{top}}\A \otimes \bigwedge^{\text{top}}T^*M_\C= (\bigwedge^{\text{top}}\C^n \otimes \bigwedge^{\text{top}}T^{1,0}X)\otimes \bigwedge^{top}T^*M_\C$. $\C^n$ is a trivial vector bundle, so we can choose a non-vanishing section $s_1$. $\bigwedge^\text{top}T^{1,0}X$ admits a (smooth) nonvanishing section $s_2$ as $X$ is symplectic, as does $\bigwedge^\text{top}TM_\C$ by dualizing $s_2$ and tensoring with $1$; let us refer to this section as $s_3$. Then $\eta = s_1 \otimes s_2 \otimes s_3$ is an integrating section for $\A$.
    
\end{exmp}

\subsection{Failure of the Kähler $b$-Sphere}
 $b$-geometry was developed largely by Melrose in 1993 \cite{Melrose} to study differential operators on manifolds with boundary. 
\begin{definition}
    A \textit{$b$-manifold} is a pair $(M,Z)$ with $M$ a compact smooth manifold, and $Z \subset M$ a codimension $1$ submanifold. By the Serre-Swan Theorem, there is a vector bundle ${}^bTM$, referred to as the \textit{$b$-tangent bundle} such that 
    \begin{equation*}
        \Gamma({}^bTM) = \{ X \in \Gamma(TM) | X \text{ is tangent to }Z \} \subset \Gamma(TM).
    \end{equation*}
\end{definition}
$b$-geometry naturally sits inside the Lie Algebroid philosophy, as the inclusion map ${}^bTM \hookrightarrow TM$ serves as an anchor for a Lie Algebroid structure, and sections possesses the same Lie Bracket as $TM$. The notion of $b$-cohomology is also defined and studied:
\begin{definition}
    Let $(M,Z)$ be a $b$-manifold, and define 
    \begin{equation*}
        {}^b\Omega^k(M,\R) := \Gamma\left( \bigwedge^k ({}^b TM)^*\right).
    \end{equation*}
    This forms a cochain complex with a natural differential ${}^bd$. We define the \textit{$b$-cohomology} of $(M,Z)$ as 
    \begin{equation*}
        {}^bH^k(M;\R) := \frac{\ker({}^bd:{}^b\Omega^k(M;\R) \ra {}^b\Omega^{k+1}(M;\R))}{\im({}^bd:{}^b\Omega^{k-1}(M;\R) \ra {}^b\Omega^{k}(M;\R))},
    \end{equation*}
\end{definition}
which is the Lie algebroid cohomology  of ${}^bTM$ with the aforementioned Lie algebroid Structure. We also have the following theorem originally in \cite[Section 2.16]{Melrose}
\begin{theorem}[Mazzeo-Melrose]
    Let $(M,Z)$ be a $b$-manifold. Then we have
    \begin{equation*}
        {}^bH^k(M) \cong H^k(M) \oplus H^{k-1}(Z).
    \end{equation*}
\end{theorem}
For example, consider the $b$-sphere, where $M=S^2$ and $Z=S^1$ is an equatorial band, which we will denote ${}^bS^2$. Here the $b$-tangent bundle has sections which are vector fields on the sphere tangent to the chosen equator. By the Mazzeo-Melrose Theorem we have that
\begin{equation*}
    H^k({}^bS^2;\R) \cong \begin{cases}
        \R \quad &\text{if }k=0,1 \\
        \R^2 &\text{if }k=2
    \end{cases}
\end{equation*}
demonstrating that Poincar\'e duality can also fail in the $b$-setting, even on a compact manifold. On this same $b$-manifold, a Kähler structure was constructed in \cite[Section 3.4]{MR4722602}, making the $b$-sphere's tangent bundle into a Kähler Lie Algebroid.
\begin{prop}
    On the Kähler $b$-sphere, the Hard Lefschetz Theorem fails.
\end{prop}
\begin{proof}
   The only map we would need to check is the map $[L]:{}^bH^0({}^bS^2;\C) \ra {}^bH^2({}^bS^2;\C)$. However by the Mazzeo-Melrose Theorem, this is a linear map between two vector spaces of differing dimensions, and cannot be an isomorphism.
\end{proof}
\begin{remark}
    In this example, the ellipticity condition needed for a Hodge Decomposition fails as the inclusion map into the complexified tangent bundle of $S^2$ is not elliptic.
\end{remark}
\begin{remark}
At times, the classical Hard Lefschetz Theorem is stated instead as saying that the maps $[L]^k$ are surjections for all $k$ rather than isomorphisms. This does not impact the result for the $b$-sphere however, as the map in question is not surjective either. 
\end{remark}
\bibliographystyle{amsplain}
\bibliography{refs}

\end{document}